\documentclass{article}
\usepackage{graphicx}
\usepackage{amsmath,amssymb,yhmath}
\usepackage{amsthm}
\usepackage[]{cleveref}
\usepackage{xcolor}
\usepackage[]{empheq}
\usepackage{array}

\usepackage{natbib}
\setcitestyle{authoryear,open={(},close={)}} 
\newtheorem{theorem}{Theorem}

\newtheorem{definition}{Definition}

\newtheorem{lemma}{Lemma}
\newtheorem{problem}{Problem}

\newcommand{\sign}{\mathrm{sign}}

\title{All paths admit trajectoids}
\author{Péter L. Várkonyi\footnote{Department of Mchanics, Materials, and Structures, Budapest University of Technology, and Economics, Hungary. E-mail: varkonyi.peter@epk.bme.hu}}

\begin{document}

\maketitle

\begin{abstract}
In a recent paper published in Nature, Y.I. Sobolev and coauthors introduced the concept of trajectoids: convex, rigid objects, which roll without slip or spin on a flat plane along a prescribed periodic, unbounded planar path. A geometric construction method applicable to many paths was introduced, and the theory was experimentally verified using objects rolling downwards on slightly inclined planes. The construction method was applicable to many but not all curves. A possible extension of the method (referred to as period-$n$ trajectoids) was also proposed, but the limits of applicability were not clarified. Here, a geometric proof is given for the existence of period-$n$ trajectoids for any sufficiently smooth prescribed curve. A somewhat different proof was recently proposed by \cite{muller2023non} independently from this work. We also highlight some related geometry problems.
\end{abstract}

\section{Background}

In a recent paper \citet{sobolev2023solid} investigated the problem of constructing a convex (possibly inhomogeneous) solid, which 
\begin{enumerate}
\item rolls without slip or spin on a flat plane $P$ in such a way that the center of mass remains at constant height above the plane.
\item and follows a prescribed periodic planar curve $T$ embedded in $P$ with discrete translation symmetry. 
\end{enumerate}
Here the \emph{curve followed by a rolling object} should be understood as the curve traced by the normal projection of the center of mass to plane $P$ as the object moves.  Slip means non-zero tangential velocity of the contact point(s) and spin means a non-zero normal component of the angular velocity.

\begin{figure} \centering
        \includegraphics[height=4cm]{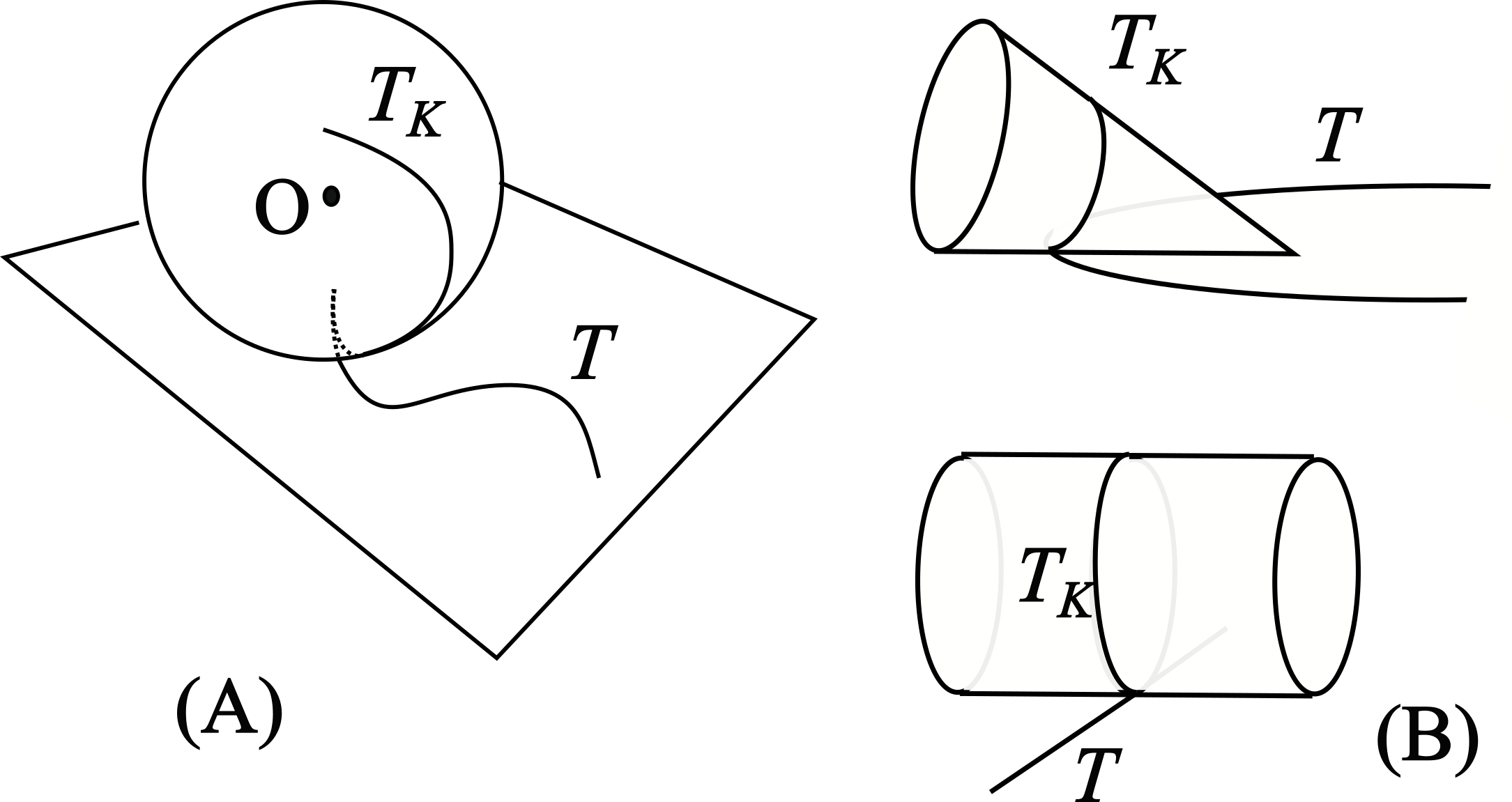}
		\caption{A: the curves $T$ and $T_K$. Right: B: Cylinders and cones are trajectoids of straight, and circular paths, respectively. }
	\label{fig:1}
\end{figure}

The construction method given by \citet{sobolev2023solid} is based on an elegant geometric observation. Consider a sphere $S_K$ of radius $K^{-1}$ with center $O$ rolling on $P$ such that the contact point follows a curve $T$ with curvature-arclength function $k(s)$. Then, the contact point traces another curve $T_K$ on the sphere, whose \emph{geodesic curvature} is also $k(s)$ (Fig. \ref{fig:1}.A). If 
\begin{itemize}
    \item $T$ is an infinite periodic curve with period $T^\star$
    \item the corresponding spherical curve $T_K$ is closed, and contains a finite number of full periods $T_K^\star$.
\end{itemize}
then one can fit tangent half-spaces to $S_K$ at all points of $T_K$ such that $S_K$ is in their interiors. The intersection of these half-spaces forms a convex object, bounded by a ruled surface possibly patched to itself along sharp edges, i.e. the boundary of the object as a whole is piecewise smooth. The object is - just like the original sphere - capable of rolling along $T$ without slip and spin indefinitely. We will refer to such objects as trajectoids of radius $K^{-1}$ corresponding to path $T$. Two simple examples where the closure condition is satisfied are illustrated by Fig. \ref{fig:1}.B. For a straight path, $T_K$ is a great circle for any $K$, and the corresponding trajectoid is an (unbounded) cylinder, whereas a circular path $T$ gives rise to a general circle-shaped $T_K$, and an (unbounded) cone-shaped trajectoid. Both examples are somewhat special. The path corresponding to  the cone does not have translation symmetry. Both trajectoids are unbounded whereas non-straight paths with translation symmetry always generate bounded surfaces. Finally, the cone surface is 'patched' at the isolated tip point due to its rotational symmetry, in contrast to most other trajectoid shapes, where the ruled surface is patched along (curved) edges. 

Ruled surfaces maintain line contact with the underlying plane, and thus their roll motion does not have a spin component as long as friction is sufficiently high and angular velocity of motion is sufficiently low (to prevent lift-off). Ruled surfaces patched together along sharp edges tend to preserve this property at low angular velocities. Moreover the  trajectoid shapes constructed above have a sharp valley of potential energy, which promotes rolling along $T$. These properties enable easy experimental realization of spin-free roll. 

In what follows, this work focuses on the kinematics aspect of the problem, rather than on the dynamics of roll motion. For simplicity it will be assumed that the object is propelled by appropriate external forces in such a way that it never stops, and it maintains line contact at all time.


The observations reviewed so far allowed \cite{sobolev2023solid} to reduce the trajectoid problem to a purely geometric problem: 

\begin{theorem}[\cite{sobolev2023solid}]
If for some path $T$ there exists $K\in\mathbb{R}^+$ such that the corresponding spherical curve $T_K$ is closed and it consists of $n\in \mathbb{N^+}$ full periods $T_K^\star$, then $T$ admits a period-$n$ trajectoid. 
\label{thm:natureMPT0}
\end{theorem}
 
Furthermore,  with the aid of
\begin{definition}
 Consider a periodic curve $T$ with translation symmetry, and a scalar $x\in\mathbb{R}^+$. If for some $K>0$, 
 \begin{itemize}
     \item [(i)] the  corresponding spherical curve $T_K^\star$ (which is unique up to isometry) with endpoints $A,B$ can be closed by the addition of two great circle segments $s_1$, $s_2$ of equal length, intersecting in point C such that $BCA\angle=x$,
     \item [(ii)] the  signed surface area $S^R_x(K)$ of the enclosed  spherical region $R_{K,x}$ is
     \begin{align}
       S^R_x(K)=\pm K^{-2}x
       \label{eq:areacondition}
     \end{align}
 \end{itemize}   
 then $T^\star$ is said to possess Property $x$.
 \label{def:x}
\end{definition}
a more compact statement is formulated:
\begin{theorem}[\cite{sobolev2023solid}]
If the path $T$ possesses Property $2\pi/n$, then it admits a period-$n$ trajectoid.
\label{thm:natureMPT}
\end{theorem}
It is easy to see that Theorem \ref{thm:natureMPT0} implies  Theorem \ref{thm:natureMPT}. If $T$ has Property $2\pi/n$ (see Fig. \ref{fig:2}) then one can rotate $R_{K,x}$ around the axis $OC$ by angles $2i\cdot\pi/n$ $(i=0,1,2,...,n-1)$ thereby creating $n$ copies of $R_{K,x}$. The union of these regions is a region of area $2K^{-2}\pi$ bounded by $n$ copies of $T_K^*$.  The boundary curves connect smoothly by the Gauss-Bonnet theorem hence the conditions of Theorem \ref{thm:natureMPT0} are fulfilled.

\begin{figure} \centering
        \includegraphics[height=4cm]{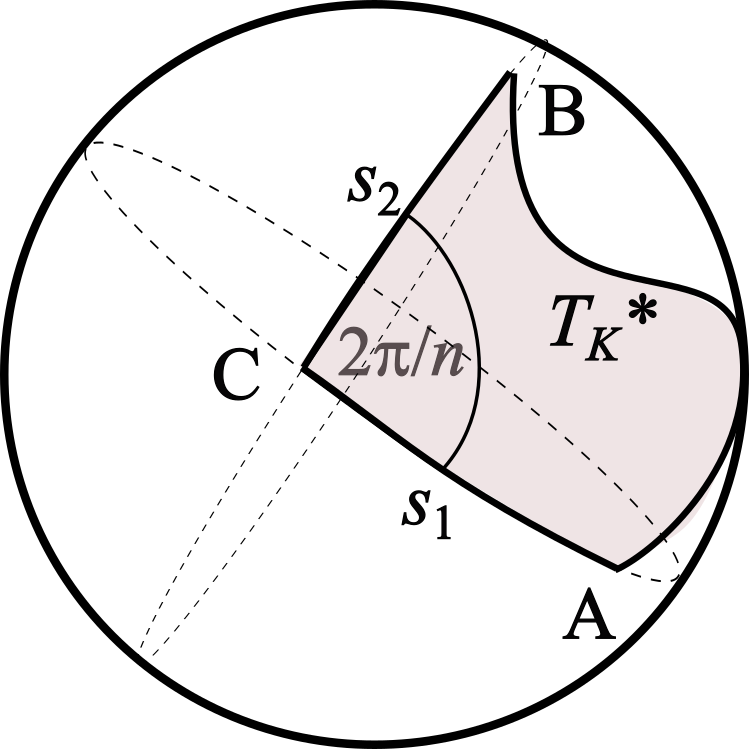}
		\caption{A sphere of radius $K^{-1}$ with the curve $T_K^\star$. The circle segments  BC and AC are of equal length, and belong to great circles (depicted as dashed curves). Property $2\pi/n$ means that the signed area of the shaded spherical region $R_{K,2\pi/n}$ equals $2K^{-2}\pi/n$. }
	\label{fig:2}
\end{figure}

It was also demonstrated by \citet{sobolev2023solid}  that most paths $T$ do not possess Property $2\pi/1$, however for $n\geq 2$, Property $2\pi/n$ can be verified numerically for many $T$ by root-finding of a scalar nonlinear equation.  The authors also conjectured that paths without Property $2\pi/n$ ($n\geq 2$) are rare, but did not prove that. Nor did they answer if there are paths that do not possess Property $2\pi/n$ for any value of $n$.
Here we prove that for any twice differentiable path, there exists a positive integer $N$, such that the path possesses Property $2\pi/n$ for all $n\geq N$, and thus all paths admit period-$n$ trajectoids.

\section{A new existence theorem of trajectoids}

\subsection{Basic idea of the new existence theorem}

We will develop a general existence theorem of trajectoids by  showing first that for sufficiently small $K,x>0$, the region $R_{K,x}$ described in point (i) of Definition \ref{def:x} exists. 
In order to investigate the area of $R_{K,x}$, it is  decomposed to two parts (which may partially overlap): 
\begin{itemize}
   \item $D_K$: the region bounded by  $T_K^\star$ and the shorter great arc $BA$, which has signed area $S^{D}(K)$.
    \item $\Delta_{K,x}$: a spherical triangle of signed area $S^\Delta_x(K)$, with vertices $A,B,C$.  
\end{itemize}
Consequently, the signed areas satisfy
\begin{align}
   S_{x}^R(K)=S^\Delta_x(K)+S^{D}(K).
   \label{eq:Ank}
\end{align}

We will also establish conditions of uniqueness of the region $D_K$, as well as bounds on the area $S^\Delta_x(K)$.  These results enable us to verify the existence of values $K$,$n$ such that \eqref{eq:areacondition} is satisfied with $x=2\pi/n$. Then, path $T$ indeed admits a trajectoid by Theorem \ref{thm:natureMPT}.

\subsection{Uniqueness of $D_K$ }
 $D_K$ is uniquely defined, unless $A,B$ are antipodal. If, A and B are antipodal for some $K$ it is easy to prove that $T$ possesses Property $2\pi/2$. However for the sake of simplicity, we will not discuss this case separately. Instead, we  require $K<K^{(1)}$ with 
$K^{(1)}=\pi L^{-1}$. Under this assumption, $A,B$ may not be antipodal points. 

\subsection{Existence of $\Delta_K$}
 
 We first establish the exact conditions of existence of $\Delta_K$:

\begin{lemma}
Assume that $A$, and $B$ are two points on a sphere of radius $K^{-1}$ with spherical distance $\wideparen{AB}=d$, and $\pi>\alpha>0$. 
\begin{itemize}
   \item If $Kd\leq \alpha$, there are two points $C_1$, $C_2$ on the sphere such that $ABC_i$ form congruent isosceles spherical triangles with $AC_i=BC_i$, $C_iAB\angle\leq\pi/2$, and  $BC_iA\angle=\alpha$.  
   \item if $Kd> \alpha$ no such point exists. 
\end{itemize}
\label{lem:tri}
\end{lemma}

\begin{proof}
 Spherical symmetry allows us to verify the first statement by  an 'inverse construction' where we first fix a point $C$ and then find the relative positions of the points $A,B$. Consider a sphere of radius $K^{-1}$, and two great circles $c_A$, $c_B$ with intersection angle $\alpha$ at point C. Mark point A, and B on $c_A$, and $c_B$ respectively such that $ACB\angle=\alpha$, and  the spherical distances of A and B from C are $\wideparen{AC}=\wideparen{BC}=\arcsin(dK\alpha^{-1})K^{-1}$. The positions of the pair $A,B$ can be  specified in two different ways, symmetrically. It can be verified by direct calculation, that  $\wideparen{AB}=d$ in both cases, which completes the proof. The second statement can be verified using proof by contradiction. One fixes $C$ and the circles $c_{A},c_{B}$. No matter how we choose A,B on the respective circles, $\wideparen{AC}=\wideparen{AB}$, and $ACB\angle=\alpha$ imply $\wideparen{AB}\leq\alpha K^{-1}$. Hence $\Delta_x$ with the desired properties does not exist.
\end{proof}

\subsection{The spherical distance of $A$ and $B$, and the area of $\Delta_{K}$}

Let $d_T(K)$ denote the spherical distance of points $A$ and $B$. According to Lemma \ref{lem:tri}, the existence of $\Delta_K$ requires that $Kd_T(K)$ is sufficiently small. In the limit of $K\to 0$, $d_T$  has a finite, positive limit $d_T(0)$, which is simply the Euclidean distance between the endpoints of the planar curve $T^\star$. 
Thus, 
\begin{align}
\left.\left[Kd_T(K)\right]\right|_{K=0}=0\label{eq:dTKK0}\\
    \left.\frac{d}{dK}\left[Kd_T(K))\right]\right|_{K=0}=d_T(0)>0
\end{align}
Under our smoothness assumptions, $d_T$ has continuous first derivative, thus there exists a value $K^{(2)}>0$ such that for any $0<K\leq K^{(2)}$,
$d_T(K)K$ is positive and monotonically increasing.

Let
\begin{align}
K^{(0)}&=\min(K^{(1)},K^{(2)})
\label{eq:K0}\\
X&= d_T(K^{(0)})K^{(0)})
\label{eq:X}
\end{align}

Then
\begin{lemma}
 For any $0<x\leq X$, there exists a unique value $0<K_x\leq K^{(0)}$ such that
 \begin{align}
     d_T(K_x)K_x=x \label{eq:dTKn}\\
     |S_{x}^\Delta(K_x)|= K_x^{-2}x,
 \label{eq:Adeltabound}
 \end{align} 
 furthermore
$\Delta_{K,x}$ exists for all $0\leq K\leq K_x$ .
\label{lem:dTK}
\end{lemma}

\begin{proof}
If $0<x\leq X$, then 
$d_T(K^{(0)})K^{(0)})\geq x$ by \eqref{eq:X},   
whereas \eqref{eq:dTKK0} implies
$d_T(0)\cdot 0< x$.
Continuity of $d_T(K)K$ implies \eqref{eq:dTKn} via the Intermediate Value Theorem. Monotonicity of $d_T(K)K$, and Lemma \ref{lem:tri} imply the existence of $\Delta_K$ if $K\leq K_x$.  Finally, The Theorem of Spherical Cosines, and the area formula of spherical triangles yield
\begin{align}
 S_{x}^\Delta(K)=
K^{-2}\left[
2\arcsin\sqrt{\frac{1+\cos(x))}{1+\cos(d_T(K)K)}}+x-\pi\right ],
\label{eq:Adelta}
\end{align}
implying 
\eqref{eq:Adeltabound}.
\end{proof}

\subsection{The new existence theorem}

We are now ready to formulate the main result of this work:

\begin{theorem}
For every twice differentiable, periodic path $T$ with translation symmetry, 
there exists a threshold $N$ such that for all $n\geq N$, $n\in\mathbb{N}$, $T$ admits a period-$n$ trajectoid. 
\end{theorem}
\begin{proof}

 According to Lemma \ref{lem:dTK}, $\Delta_{K_x,x}$ exists for all $x\leq X$, moreover there are two points $C_1$, $C_2$ on the sphere, which can be used to construct the region $\Delta_{K_x,x}$ according to Lemma \ref{lem:tri}. The corresponding two spherical triangles are mirror images of one another with respect to the $OAB$ plane, hence their signed spherical areas have the same magnitudes but opposite signs. 
 
 If $S^D(K_x)\neq 0$, then it is assumed without loss of generality that the signed area of $BC_1A$ satisfies
 \begin{align}
     \sign (S^D(K_x)) =\sign (S^\Delta_x(K_x)).
     \label{eq:sign}
 \end{align}
 We use point $C_1$ to construct the region $\Delta_{K,x}$ for each $0\leq K \leq K_x$. Then, \eqref{eq:Ank}, \eqref{eq:Adeltabound}, and \eqref{eq:sign} imply
$$
|S_{x}^R(K_x)|\geq K_x^{-2}x.
$$
At the same time we have seen that $S^D(0)$, $S^\Delta (0)$ are finite, i.e.
$$
|S_{x}^R(0)|=|S^D(0)+S^\Delta_x(0) |< \lim_{K\to 0} K^{-2}x=\infty .
$$
By continuity of $|S_x^R(K)|$, there exists a scalar $0< Q_x < K_x$ for which \begin{align}
  |S_x^R(Q_x)|=Q_x^{-2}x. 
  \label{eq:AxQx}
\end{align}

Let $n\geq2\pi/X$ be an arbitrary integer, $x=2\pi/n$, and $K=Q_x$. Then, \eqref{eq:AxQx} implies that the condition \eqref{eq:areacondition} is satisfied. Hence $T$ possesses Property $2\pi/n$, and a period-$n$ trajectoid of inverse radius $Q_x$ exists by Theorem \ref{thm:natureMPT}.
\end{proof}


\section{Some related problems}

 Our work sheds light on some intriguing geometric problems:
 \begin{problem}
Consider a finite, planar curve $T^\star$, Let $T_K^\star$ denote another curve on a sphere of radius $K^{-1}$ with the same geodesic curvature and let $d_T(K)$ be the spherical distance of its endpoints. Then
\begin{enumerate}
    \item How van one determine the largest value of the function $d_T(K)K$?
    \item How van one determine the largest value of $x$ for which $T$ possesses Property $x$?
    \item Are these values equal?
\end{enumerate}
\end{problem} 
We have seen that both quantities play crucial roles in the construction of trajectoids. However the first question has other interpretations, too. For example, 
imagine that an airplane departs from the North Pole of a spherical planet with unknown radius $K^{-1}$. It follows a prescribed itinerary $T^*$, i.e. a finite path with prescribed length and geodesic curvature function. Which are those latitudes of the planet, which can be reached at the end of this journey for some value of $K$?

We also mention an interesting property of trajectoids. Closed spherical curves $T_K$ corresponding to a trajectoid (cf. Theorem \ref{thm:natureMPT0})  cut the  sphere to two parts of equal area \citep{segerman2021rolling,sobolev2023solid}. Every area-halving curve is the derivative of a family of arbitrary closed spherical curves \citep{arnol1995geometry}. What does this property say about the motion of trajectoids? It is easy to see that roll motion is balanced in the following sense. Imagine that a trajectoid is filled with a highly viscous liquid suspension containing heavy particles, such that the particles continuously sink vertically with constant velocity. The cited result of \citep{arnol1995geometry} implies that the velocity profile of rolling can be chosen in such a way that this suspension remains mixed indefinitely, i.e. the particles do not accumulate on one side of the container. 
\bibliography{citations}

\end{document}